\documentclass[12pt]{amsart}
\usepackage[english]{babel}
\usepackage[latin1]{inputenc}
\usepackage{fancyhdr}
\usepackage{amsmath,amssymb,amsthm,graphicx}
\usepackage{latexsym,ifpdf}
\usepackage{enumitem}
\usepackage{appendix}
\usepackage[T1]{fontenc}
\usepackage{babel}

\usepackage{color}

\newtheorem{theorem}{Theorem}
\newtheorem{corollary}[theorem]{Corollary}
\newtheorem{definition}[theorem]{Definition}
\newtheorem{lemma}[theorem]{Lemma}

\usepackage{color}

\begin{document}

\title{Commutator estimates in $W^*$-factors}

\author{A. F. Ber}
\address{Department of Mathematics, Tashkent State University, Uzbekistan }
\email{ber@ucd.uz}

\author{F. A. Sukochev}
\address{School of Mathematics and Statistics, University of New South Wales, Sydney, NSW 2052, Australia }
\email{f.sukochev@unsw.edu.au}


\bigskip
\begin{abstract}
\noindent Let $\mathcal{M}$ be a $W^*$-factor and let $S\left( \mathcal{M}
\right) $ be the space of all measurable operators affiliated with $\mathcal{M}$. It is shown that
for any self-adjoint element $a\in S(\mathcal{M})$ there exists a scalar
$\lambda_0\in\mathbb{R}$, such that for all $\varepsilon
> 0$, there exists a unitary element $u_\varepsilon$ from $\mathcal{M}$,
satisfying $|[a,u_\varepsilon]| \geq
(1-\varepsilon)|a-\lambda_0\mathbf{1}|$. A corollary of this
result is that for any derivation $\delta$ on  $\mathcal{M}$ with
the range in an ideal $I\subseteq\mathcal{M}$, the derivation
$\delta$ is inner, that is $\delta(\cdot)=\delta_a(\cdot)=[a,\cdot]$, and $a\in I$.
Similar results are also obtained for inner
derivations on $S(\mathcal{M})$.
\end{abstract}



\maketitle

\section{Introduction}

Let $\mathcal{M}$ be a $W^*$-algebra and $\mathcal{N}$
its $W^*$-subalgebra, let $I$ be an ideal in
$\mathcal{M}$ and let $\delta$ be a derivation on $\mathcal{N}$ with the range in an ideal $I$.
The problem studied in \cite{JP,KW,PR} can be stated as follows: \textit{What are the conditions on
$\mathcal{M}$, $\mathcal{N}$ and $I$ which guarantee that $\delta(\cdot) = \delta_a(\cdot):= [a,\cdot]$, where $a\in I$?}
In the present article, we show that the answer is affirmative when $\mathcal{N}=\mathcal{M}$ is an arbitrary $W^*$-factor and $I$ is an arbitrary ideal in $\mathcal{M}$ (see Corollaries \ref{c1},\ref{c11}). Our methods are completely different from the methods employed in  \cite{JP,KW,PR} and are strong enough to enable us (see Corollaries \ref{c2},\ref{c3}) to also treat an analogous question in a much more general setting of the theory of non-commutative integration on von Neumann algebras, initiated by I.E. Segal \cite{Se} (for alternative approach to this theory, see E. Nelson's paper \cite{Ne}). All necessary definitions will be given in the next section.

Recall that the classical algebras of measurable operators associated with a von Neumann algebra $\mathcal{M}$ and/or with a pair $\left( \mathcal{M},\tau \right) $ consisting of a semi-finite von Neumann algebra $\mathcal{M}$ and a faithful normal semi-finite trace $\tau$ are the following:
\begin{enumerate}
\item[(i)] the space of all measurable operators $S\left( \mathcal{M}
\right) $ \cite{Se};

\item[(ii)] the space $S\left( \mathcal{M},\tau
\right) $ of all $\tau $-measurable operators \cite{Ne}.
\end{enumerate}

\noindent It should be noted that  we always have $S\left( \mathcal{M},\tau\right)  \subseteq S\left( \mathcal{M}\right)$, but in the important case when
 $\mathcal{M}$ is a semi-finite factor (respectively, of type $I$ or $III$), we have $S\left( \mathcal{M}
\right)=S\left( \mathcal{M},\tau\right)$ (respectively, $S\left( \mathcal{M},\tau\right)=\mathcal{M}$).

Our main result in this paper is the following theorem.

\begin{theorem}
\label{main} Let $\mathcal{M}$ be a $W^*$-factor and let $a=a^*\in S(\mathcal{M})$.
\begin{enumerate}[label={(\roman*).}]
  \item If $\mathcal{M}$ is a finite factor or else a purely infinite $\sigma$-finite factor, then there exists $\lambda_0\in\mathbb{R}$ and
$u_0=u_0^*\in U(\mathcal{M})$, such that
\begin{equation}\label{first main}|[a,u_0]|=u_0^*|a-\lambda_0\mathbf{1}|u_0 +
|a-\lambda_0\mathbf{1}|,
\end{equation}
where $U(\mathcal{M})$ is a group of all unitary operators in
$\mathcal{M}$;
    \item There exists $\lambda_0\in\mathbb{R}$, so that for any $\varepsilon > 0$ there exists $u_\varepsilon=u_\varepsilon^* \in
U(\mathcal{M})$ such that \begin{equation}\label{second main}
|[a,u_\varepsilon]| \geq (1-\varepsilon)|a-\lambda_0\mathbf{1}|.
\end{equation}
\end{enumerate}
If $\mathcal{M}$ is an infinite semi-finite $\sigma$-finite factor, then the result stated in (ii) above is sharp. More precisely, in this case
there exists $0\leq a\in S(\mathcal{M})$ such that for all
$\lambda\in\mathbb{C}$ and all $u\in U(\mathcal{M})$ the inequality
$|[a,u]|\geq |a-\lambda\mathbf{1}|$ fails.
\end{theorem}

\section{Preliminaries}

For details on von Neumann algebra theory, the reader is referred
to e.g. \cite{Dix}, \cite{KR2}, \cite{Sak} or \cite{Tak}. General
facts concerning measurable operators may be found in \cite{Ne},
\cite{Se} (see also \cite[Chapter IX]{Ta2}). For the convenience
of the reader, some of the basic definitions are recalled.

Let $\mathcal{M}$ be a von Neumann algebra on a Hilbert space $H$
equipped with a semi-finite normal faithful trace $\tau$. The set
of all self-adjoint projections (respectively, all unitary
elements) in $\mathcal{M}$ is denoted by
$P\left(\mathcal{M}\right)$ (respectively,
$U\left(\mathcal{M}\right)$). The algebra $B(H)$ of all bounded
linear operators on $H$ is equipped with its standard trace $Tr$.
The commutant of a set $D\subset B(H)$ is denoted by $D^{\prime
}$. We use the notation $s(x),l(x),r(x)$ to denote the support,
left support, right support respectively of an element
$x\in\mathcal{M}$.

Let $p,q\in P\left(\mathcal{M}\right)$. The projections
$p$ and $q$ are said to be equivalent, if there exists
a partial isometry $v\in\mathcal{M}$, such that $v^*v=p,\ vv^*=q$. In this case, we write $p\sim q$.
The fact that the projections $p$ and $q$ are not equivalent is recorded as $p\nsim q$.
If there exists a projection $q_1\in P\left(\mathcal{M}\right)$ such that $q_1\leq p,\ q_1\sim q$, then we write
$q\preceq p$. If $q\preceq p$ and $p\nsim q$, then we employ the notation $q\prec p$.


A linear operator $x:\mathfrak{D}\left( x\right) \rightarrow
H $, where the domain $\mathfrak{D}\left( x\right) $ of $x$ is a linear
subspace of $H$, is said to be {\it affiliated} with $\mathcal{M}$ if $yx\subseteq
xy$ for all $y\in \mathcal{M}^{\prime }$ (which is denoted by $x\eta
\mathcal{M}$). A linear operator $x:\mathfrak{D}\left( x\right) \rightarrow
H $ is termed {\it measurable} with respect to $\mathcal{M}$ if $x$ is closed,
densely defined, affiliated with $\mathcal{M}$ and there exists a sequence $%
\left\{ p_{n}\right\} _{n=1}^{\infty }$ in $P\left( \mathcal{M}\right) $
such that $p_{n}\uparrow \mathbf{1}$, $p_{n}\left( H\right) \subseteq
\mathfrak{D}\left( x\right) $ and $p_{n}^{\bot }$ is a finite projection
(with respect to $\mathcal{M}$) for all $n$. It should be noted that the
condition $p_{n}\left( H\right) \subseteq \mathfrak{D}\left( x\right) $
implies that $xp_{n}\in \mathcal{M}$. The collection of all measurable
operators with respect to $\mathcal{M}$ is denoted by $S\left( \mathcal{M}%
\right) $, which is a unital $\ast $-algebra with respect to strong sums and
products (denoted simply by $x+y$ and $xy$ for all $x,y\in S\left( \mathcal{M%
}\right) $).

Let $a$ be a self-adjoint operator affiliated with $\mathcal{M}$.
We denote its spectral measure by $\{e^a\}$. It is known if $x$ is
a closed operator in $H$ with the polar decomposition $x = u|x|$
and $x\eta \mathcal{M}$, then $u\in\mathcal{M}$ and $e\in
\mathcal{M}$ for all projections $e\in \{e^{|x|}\}$. Moreover,
$x\in S(\mathcal{M})$ if and only if $x$ is closed, densely
defined, affiliated with $\mathcal{M}$ and $e^{|x|}(\lambda,
\infty)$ is a finite projection for some $\lambda> 0$. It follows
immediately that in the case when $\mathcal{M}$ is a von Neumann
algebra of type $III$ or
 a type $I$ factor, we have $S(\mathcal{M})= \mathcal{M}$. For type $II$ von Neumann algebras, this is no longer true.

An operator $x\in S\left( \mathcal{M}\right) $ is called $\tau
$-{\it measurable} if there exists a sequence
$\left\{p_{n}\right\} _{n=1}^{\infty }$ in $P\left(
\mathcal{M}\right) $ such that $p_{n}\uparrow \mathbf{1}$,
$p_{n}\left( H\right) \subseteq \mathfrak{D}\left( x\right) $ and
$\tau \left( p_{n}^{\bot }\right) <\infty $ for all $n$. The
collection $S\left( \tau \right) $ of all $\tau $-measurable
operators is a unital $\ast $-subalgebra of $S\left(
\mathcal{M}\right) $ denoted by $S\left( \mathcal{M}, \tau\right)
$. It is well known that a linear operator $x$ belongs to $S\left(
\mathcal{M}, \tau\right) $ if and only if $x\in S(\mathcal{M})$
and there exists $\lambda>0$ such that $\tau(e^{|x|}(\lambda,
\infty))<\infty$.


In this paper, we shall frequently assume that  $\mathcal{M} $ is
a factor. If $\mathcal{M}$  is a semi-finite factor with the trace
$\tau $, then the notions of $\tau$-finite and (algebraically)
finite projections coincide. An immediate corollary of this
observation is that, the algebras $S(\mathcal{M})$ and
$S(\mathcal{M},\tau)$ coincide in this setting.

\section{The proof of Theorem \ref{main}}

For better readability, we break the theorem's proof into the following series of
lemmas.

\begin{lemma}
\label{l00} Let $p,q,r\in P(\mathcal{M})$, $p<q$, $p\prec r\prec
q$. Then there exists $r_1\in P(\mathcal{M})$, such that $r_1\sim
r$ and $p<r_1<q$.
\end{lemma}
\begin{proof}
There exists $p_1\in P(\mathcal{M})$, such that  $p\sim p_1<r$.
Assume that $r-p_1\succeq q-p$. Then $r=(r-p_1)+p_1\succeq
(q-p)+p=q$, which contradicts our assumption. Therefore $r-p_1\prec
q-p$. Hence, there exists $p_2\in P(\mathcal{M})$, such that
$r-p_1\sim p_2<q-p$. Then $p<p+p_2<q$ and $p+p_2\sim
p_1+(r-p_1)=r$. Setting $r_1 = p+p_2$ completes the proof.
\end{proof}

\begin{lemma}
\label{l0} Let $p$ be an infinite projection in $\mathcal{M}$. Then:
\begin{enumerate}[label={(\roman*).}]
    \item If $P(\mathcal{M}) \ni q_1,...,q_n,... \preceq p$,
$q_nq_m=0$ for all $n\neq m$, then $\bigvee_{n=1}^\infty q_n \preceq p$.
    \item If $P(\mathcal{M}) \ni q_1,...,q_n\prec p$, $q_iq_j=0$ for all
$i\neq j$, then $\bigvee_{i=1}^n q_i \prec p$.
    \item If $p\succeq \mathbf{1}-p$, then $p\sim \mathbf{1}$.
    \item If $P(\mathcal{M}) \ni q\prec p,\ qp=pq$, then
$p(\mathbf{1}-q)\sim p$.
\end{enumerate}
\end{lemma}
\begin{proof}
(i). Since $p$ is an infinite projection, there exist pairwise disjoint projections $p_1,...,p_n,...\in P(\mathcal{M})$, such that $p=\bigvee_{n=1}^\infty p_n$, $p_n\sim p$ for all $n\in
\mathbb{N}$. Then $p_n\succeq q_n$ for all $n\in \mathbb{N}$.
Hence $p=\bigvee_{n=1}^\infty p_n\succeq \bigvee_{n=1}^\infty
q_n$.

(ii). Since $\mathcal{M}$ is a factor, every projection is comparable to
every other projection. Thus reordering if necessary, we may
assume that $q_1\preceq q_2\preceq ... \preceq q_n$. If $q_n$ is a
finite projection, then $\bigvee_{i=1}^n q_i$ is a finite and
$\bigvee_{i=1}^n q_i\prec p$. If $q_n$ is infinite, then by (i),
$\bigvee_{i=1}^n q_i \preceq q_n \prec p$.

(iii). Since $\mathbf{1}=p+(\mathbf{1}-p)$, it follows from (i) that $\mathbf{1}\preceq p$.

(iv). We have $p=qp+(\mathbf{1}-q)p,\ qp\leq q \prec p$. If it were true that $(\mathbf{1}-q)p\prec p$, then by (ii) we have $p\prec p$, which is false. Thus $(1-q)p = 1-pq \succeq p$ and certainly $(1-q)p \preceq p$. The result follows immediately.
\end{proof}

In the special case when $\mathcal{M}$ is semifinite and $a$ is positive, it may be of interest to compare the result given below with \cite[Theorem 3.5]{DDP} and \cite[Lemma 4.1]{CS}.

\begin{lemma}
\label{l1} Let $a\in S_h(\mathcal{M})$ and $p,q\in P(\mathcal{M})$, $p\succeq q$. Suppose that one of the following
conditions holds:
\begin{enumerate}[label={(\roman*).}]
    \item $q$ is finite and there exists a sequence of finite projections $\{p_n\}$ in $\mathcal{M}$ such that $p_n \uparrow p$ and $ap_n=p_n a$ for all $n\in \mathbb{N}$;
    \item $q$ is an infinite projection and $ap=pa\in \mathcal{M}$.
\end{enumerate}
Then there exists a projection $q_1\in P(\mathcal{M})$ such that $q_1\sim q$, $aq_1=q_1a$ and such that $q_1\leq p$.
\end{lemma}
\begin{proof}

Assume (i) holds. By the assumption $\mathcal{M}$ contains finite
projections and therefore $\mathcal{M}$ is a factor of type $I$ or
else of type $II$. Therefore $\mathcal{M}$ admits a faithful
normal semifinite trace $\tau$. Let $D$ be a commuting family
given by the spectral measure $\{e^a\}$ and let
$\mathcal{A}_1:=D^\prime\cap \mathcal{M}$. Since $ap_n=p_n a$ for
all $n\in \mathbb{N}$ and $p_n \uparrow p$, we also have $ap=pa$.
Therefore $p\in\mathcal{A}_1$. Then
$\mathcal{A}:=p\mathcal{A}_1p=\mathcal{A}_1p$ is a
$W^*$-subalgebra in $\mathcal{M}$ with the unit $p$. Let $e$ be an
atom in $\mathcal{A}$ and let $f\in P(\mathcal{M})$ be such that
$f<e$. Then for every $t\in \{e^a\}$ we have
$tp=pt\in P(\mathcal{A})$ and so $tf=t(p(ef))=((tp)e)f\in
\{0,e\}f=\{0,ef\}=\{0,f\}$, that is $tf=ft$. Therefore $f\in
P(\mathcal{A})$ and since $e$ is an atom in $\mathcal{A}$ we
conclude that $f=0$. Therefore $e$ is also an atom in
$\mathcal{M}$.

In the set $P(\mathcal{A})$ we select the subset $M(q)=\{r\in
P(\mathcal{A}):\ \tau(r)\geq \tau(q)\}$. If $\tau(q)=\tau(p)$,
then $q_1:=p\sim q$ and the proof is finished. Therefore, we
assume below that $\tau(q)<\tau(p)$. Observing that
$p_n\in\mathcal{A}_1$, $n\ge 1$ and that $\tau(q)<\infty$,
$\tau(p)>\tau(q)$, $\tau(p_n)\uparrow \tau(p)$, we see that there
exists $n\ge 1$, such that $\tau(p_n)>\tau(q)$. This shows, in
particular, that $M(q)$ is a not empty. Let $\mathfrak{C}$ be a
linearly ordered family in $M(q)$. Then the mapping
$\tau|_{\mathfrak{C}}$ into the interval $[\tau(q),\tau(p)]$ is
injective and order preserving. Since the trace $\tau$ is normal, we
have
$\tau(\bigwedge\mathfrak{C})=\bigwedge_{r\in\mathfrak{C}}\tau(r)\geq\tau(q)$.
Therefore $\bigwedge\mathfrak{C}\in M(q)$. This shows that the set
$M(q)$ satisfies Zorn's lemma assumption and therefore it has a
minimal element. Let $r_0$ be a minimal element in $M(q)$.
 If $\tau(r_0)=\tau(q)$, then we set $q_1:=r_0\sim
q$ and the proof is finished.  Suppose that $\tau(r_0)>\tau(q)$.
Moreover, consider the set $N(q)=\{r\in P(\mathcal{A}):\ \tau(r)\leq
\tau(q),\ r\leq r_0\}$. This set is not empty, in particular $0\in
N(q)$. Arguing as above, we see that $N(q)$ has a maximal element
$r_1$. We claim that $\tau(r_1)<\tau(q)$. Indeed, if it is not so,
then $\tau(r_1)=\tau(q)$ and $r_0>r_1\in M(q)$, which contradicts
the assumption that $r_0$ is minimal. Thus, $\tau(r_1)<\tau(r_0)$,
that is $r_0-r_1>0$. Observe that $r_0-r_1$ is an atom in
$\mathcal{A}$. Indeed, if there exists $0<f<r_0-r_1,\ f\in
P(\mathcal{A})$, then either $\tau(r_1)<\tau(r_1+f)\leq \tau(q)$
(which contradicts the assumption that $r_1$ is maximal), or else
$\tau(r_0)>\tau(r_1+f)\geq \tau(q)$, which contradicts the
assumption that $r_0$ is minimal. Thus, $r_0-r_1$ is an atom in
$\mathcal{A}$ and hence, as we already observed above, it is also
an atom in $\mathcal{M}$. On the other hand, it follows from Lemma
\ref{l00} that there exists $r_2\in P(\mathcal{M})$, such that
$r_2\sim q,\ r_1<r_2<r_0$. In particular, $0<r_2-r_1<r_0-r_1$,
that is $r_0-r_1$ is a not an atom. We have arrived at the
contradiction. Therefore, $q_1=r_0\sim q$.

Assume (ii) holds. By the assumption there exists a projection
$q_1^0\in \mathcal{M}$, such that $q_1^0\leq p$ and $q_1^0\sim q$.
We set $q_1^n:=l(a^n q_1^0)$ for all $n>0$,
$q_1:=\bigvee_{k=0}^\infty q_1^k$. We claim that $q_1\sim q$.
Indeed, since $q_1\ge q_1^0\sim q$, we
have $q_1\succeq q$. On the other hand, we have $q_1^n\sim
r(a^nq_1^0)\leq q_1^0\sim q$, which implies $q_1^n \preceq q$ for
all $n\geq 0$. Now, we shall show that in fact $q_1 \preceq q$.
Note that although  $q$ is an infinite projection we cannot simply refer to Lemma \ref{l0}(i) since the sequence
$\{q_1^k\}_{k\ge 0}$ does not necessarily consist of pairwise orthogonal elements.
However, representing the projection $q_1$ as
\begin{align*}
q_1&=\bigvee_{k=0}^\infty q_1^k=\sum_{m=1}^\infty(\bigvee_{k=0}^m
q_1^k-\bigvee_{k=0}^{m-1} q_1^k)+q_1^0
=\sum_{m=1}^\infty(q_1^m\vee \bigvee_{k=0}^{m-1}
q_1^k-\bigvee_{k=0}^{m-1} q_1^k) +q_1^0.
\end{align*}
and noting that $q_1^0\sim q$ and
\begin{align*}
q_1^m\vee\bigvee_{k=0}^{m-1} q_1^k-\bigvee_{k=0}^{m-1} q_1^k\sim
q_1^m
-(q_1^m\wedge \bigvee_{k=0}^{m-1} q_1^k) &\leq q_1^m \preceq q
\end{align*}
we infer via Lemma \ref{l0}(i) that $q_1 \preceq q$.  This completes the proof of the claim.

Since $ap=pa$ and
$q_1^0 \leq p$, we have $pa^n q_1^0 = a^n p q_1^0=a^n q_1^0$, and
so $q_1^n\leq p$ for all $n>0$. Hence, $q_1\leq p$. It remains to
show that $aq_1=q_1a$. The subspace $q_1(H)$ coincides with the
closure of linear span of the set $Q:=\{ a^n q_1^0(H):\ n>0\}$. By
the assumption the operator $ap$ is bounded, and since $q_1\leq
p$, the operator $aq_1$ is also bounded. Thus, for every vector
$\xi\in Q$, the vector $a\xi=aq_1\xi$ again belongs to $Q$. Again
appealing to the fact that $aq_1$ is bounded, we infer $q_1 a
q_1=a q_1$. From this we conclude that $aq_1=q_1 a$.
\end{proof}

\begin{lemma}
\label{l2} Let $\mathcal{M}$ be an infinite factor and let
$e^a(-\infty,0]\prec e^a(0,+\infty)$, $e^a(0,\lambda]\succ
e^a(-\infty,0]$ and $e^a(0,\lambda]\succ e^a(\lambda,+\infty)$ for
all $\lambda>0$. Then for all $\varepsilon>0$ there exists
$u_\varepsilon=u_\varepsilon^*\in U(\mathcal{M})$, such that
$|[a,u_\varepsilon]|\geq (1-\varepsilon)|a|$.
\end{lemma}
\begin{proof}
Certainly the result is trivial for $\varepsilon \geq 1$ and so we
restrict ourselves to the case of $\varepsilon<1$. Our aim is to
build a decreasing sequence of positive scalars
$\{\lambda_n\}_{n=0}^\infty$ converging to zero and two sequences
$\{p_n\}_{n=0}^\infty,\ \{q_n\}_{n=0}^\infty$ of pairwise
orthogonal projections from $\mathcal{M}$, which satisfy:
\begin{enumerate}[label={(\roman*).}]
    \item $p_nq_m=0,\ ap_n=p_na,\ aq_n=q_na,\ p_n\sim q_n$ for all $n,m\geq 0$.
    \item $p_n\leq e^a(\lambda_n,+\infty),\ q_n\leq
e^a(-\infty,\varepsilon\lambda_n]$ for all $n\geq 0$ and $q_0\geq
e^a(-\infty,0]$.
    \item $\bigvee_{n=0}^\infty p_n \vee \bigvee_{n=0}^\infty q_n =
\mathbf{1}$.
\end{enumerate}
Consider the three cases:

(a). Suppose that the projection $e^a(-\infty,0]$ and all the
projections $e^a(\lambda,+\infty)$ for all  $\lambda>0$ are
finite. Then $e^a(0,\lambda]$ is the supremum of an increasing
sequence of finite projections $\{e^a(\lambda/n,\lambda]\}_{n\ge
1}$ for all $\lambda>0$. We claim that there exists $\lambda_0>0$
such that $e^a(\lambda_0,+\infty)\succ e^a(-\infty,0]$. Indeed,
$e^a(0,+\infty)=\bigvee_{n=1}^\infty e^a(1/n,+\infty)$, where, by
the assumption every projection $e^a(1/n,+\infty)$, $n\ge 1$ is
finite. Therefore, if it were $e^a(1/n,+\infty)\preceq
e^a(-\infty,0]$ for all $n\in\mathbb{N}$, we would have then
$e^a(0,+\infty)\preceq e^a(-\infty,0]$ (see e.g. \cite[Chapter V, Lemma 2.2]{Tak}), which is not the case. Thus, our claim holds and there
exists $r\in\mathcal{M}$ such that $e^a(-\infty,0]\sim
r<e^a(\lambda_0,+\infty)$. Now, we claim the existence of a
converging to zero sequence $\{\lambda_n\}_{n=0}^\infty$ of
positive numbers such that $e^a(\lambda_{n+1},+\infty)\succ
e^a(\lambda_n,+\infty)$ for all $n\geq 0$. This claim is justified
by the same argument as above: since $e^a(0,+\infty)$ is an
infinite projection and since $e^a(1/n,+\infty)\uparrow
e^a(0,+\infty)$ for $n\rightarrow \infty$, we see that for any
finite projection $q\in \mathcal{M}$ there exists $n\ge 1$, such
that $e^a(1/n,+\infty)\succ q$ (again by \cite[Chapter V, Lemma 2.2]{Tak}).
Indeed, if it were that $e^a(1/n,+\infty)\preceq q$, for all $n\ge
1$, we would then have to $e^a(0,+\infty)\preceq q$, which is
false, since the projection $e^a(0,+\infty)$ is infinite, whereas
the projection $q$ is finite.

(b). Suppose that the projection $e^a(-\infty,0]$ is finite and
there exists a number $\lambda_0>0$, such that
$e^a(\lambda_0,+\infty)$ is an infinite projection. Then there
exists a projection $r\in\mathcal{M}$ such that
$e^a(-\infty,0]\sim r<e^a(\lambda_0,+\infty)$. In addition,
$e^a(\lambda_0,+\infty)-r$ is an infinite projection.

(c). Suppose that the projection $e^a(-\infty,0]$ is infinite.
Then there exists a scalar $\lambda_0>0$, such that
$e^a(-\infty,0]\prec e^a(\lambda_0,+\infty)$. Indeed if the
opposite inequality were to hold  for every $\lambda>0$, then
Lemma \ref{l0} (i) would yield the estimate
\begin{align*}
e^a(0,+\infty)&=\sum_{k=1}^\infty
e^a(1/(k+1),1/k]+e^a(1,+\infty)\preceq e^a(-\infty,0],
\end{align*}
which contradicts the assumption  $e^a(-\infty,0]\prec
e^a(0,\lambda]\leq e^a(0,+\infty)$ for any $\lambda>0$. Therefore
there exists a projection $r\in\mathcal{M}$, such that
$e^a(-\infty,0]\sim r<e^a(\lambda_0,+\infty)$ and
$e^a(\lambda_0,+\infty)-r$ is an infinite projection.

In all these cases, let us set $p_0:=e^a(\lambda_0,+\infty)$.
Since, by the assumption, $e^a(0,\varepsilon\lambda_0]\succ
e^a(\varepsilon\lambda_0,+\infty)\geq e^a(\lambda_0,+\infty)$, we
have $e^a(0,\varepsilon\lambda_0]\succ p_0$, it follows from Lemma
\ref{l1}(i) in the case (a) and from Lemma \ref{l1}(ii) in the
cases (b) and (c) that there exists a projection
$q^1_0\in\mathcal{M}$, for which $p_0-r\sim q^1_0<
e^a(0,\varepsilon\lambda_0]$ and $aq^1_0=q^1_0 a$. Let us set
$q_0:=e^a(-\infty,0]+q^1_0$. Since $r \sim e^a(-\infty,\lambda)$
in all three cases, we have $q_0\sim p_0$.

Now, similar to the case (a), we shall show that in the cases of
(b) and (c) there also exists a decreasing sequence of positive
real numbers $\{\lambda_n\}_{n=0}^\infty$, which converges to $0$
and such that $e^a(\lambda_{n+1},+\infty)\succ
e^a(\lambda_n,+\infty)$ for all $n\geq 0$. To this end, it is
sufficient to show that for every $\lambda >0$ the inequality
$e^a(t,+\infty)\sim e^a(\lambda,+\infty)$ for all $t\in
(0,\lambda)$ does not hold. Suppose the opposite, and let a scalar
$\lambda$ be such that for all $t\in (0,\lambda)$, we have
$e^a(t,+\infty)\sim e^a(\lambda,+\infty)$. Then we have
$e^a(\lambda/(k+1),\lambda/k]\leq e^a(\lambda/(k+1),+\infty)\sim
e^a(\lambda,+\infty)$ for every $k\ge 1$, that is
$e^a(\lambda/(k+1),\lambda/k]\preceq e^a(\lambda,+\infty)$ and so
by Lemma \ref{l0} (i), it follows
\begin{align*}
e^a(0,\lambda]=\sum_{k=1}^\infty
e^a(\lambda/(k+1),\lambda/k]\preceq e^a(\lambda,+\infty).
\end{align*}
However, this contradicts our initial assumption that
$e^a(0,\lambda]\succ e^a(\lambda,+\infty)$.

Now, we are well equipped to proceed with the construction of the sequences $\{p_n\}_{n=0}^\infty,\ \{q_n\}_{n=0}^\infty$.

Suppose the projections $p_0,...,p_n;\ q_0,...,q_n$ have already been constructed. Set
\begin{align*}
p_{n+1}=e^a(\lambda_{n+1},+\infty)
\prod_{k=0}^n(\mathbf{1}-p_k)\prod_{k=0}^n(\mathbf{1}-q_k).
\end{align*}
In the case (a), all the projections $p_k,\ q_k$ for $k\leq n$ are
finite and $e^a(0,\varepsilon\lambda_{n+1}]$ is an infinite
projection. Hence, $e^a(0,\varepsilon\lambda_n]
\prod_{k=0}^n(\mathbf{1}-p_k) \prod_{k=0}^n(\mathbf{1}-q_k)$ is an
infinite projection  for all $n\ge 1$. We shall now explain to the
reader that we are now in a position to apply Lemma \ref{l1}(i)
and infer that there exists a projection $q_{n+1}\in\mathcal{M}$,
such that
\begin{align*}
p_{n+1}\sim q_{n+1}<e^a(0,\varepsilon\lambda_{n+1}]
\prod_{k=0}^n(\mathbf{1}-p_k)\prod_{k=0}^n(\mathbf{1}-q_k),
\end{align*}
and $aq_{n+1}=q_{n+1}a$. To see that Lemma \ref{l1}(i) is indeed
applicable, set $p:=e^a(0,\varepsilon\lambda_{n+1}]
\prod_{k=0}^n(\mathbf{1}-p_k)\prod_{k=0}^n(\mathbf{1}-q_k)$ and
$q:=p_{n+1}$. Observe, that here $p$ is infinite and $q$ is
finite, in particular $p\succ q$. The role of finite projections
$p_m$'s from that lemma is then played by the sequence
$\{e^a(1/m,\varepsilon\lambda_{n+1})\}_{m\ge 1}$. Observe that $e^a(1/m,\varepsilon\lambda_{n+1})\uparrow_m
e^a(0,\varepsilon\lambda_{n+1})$ and this sequence obviously commutes with the operator $a$. This completes the
construction in the case (a).

Now let us consider the cases (b) and (c). Since $p_k\leq
e^a(\lambda_k,+\infty)$, we have
$$
\sum_{k=0}^n p_k = \bigvee_{k=0}^n p_k \leq \bigvee_{k=0}^n
e^a(\lambda_k,+\infty) = e^a(\lambda_n,+\infty)
$$
and so
\begin{align*}
\sum_{k=0}^n p_k \leq e^a(\lambda_n,+\infty) \prec
e^a(\lambda_{n+1},+\infty),\ n\ge 1.
\end{align*}
Since, $q_k\sim p_k \prec e^a(\lambda_{n+1},+\infty)$ for all
$k=0,1,2,\dots n$, we obtain, via an application of Lemma \ref{l0}
(ii), that
\begin{align*}
\sum_{k=0}^n p_k + \sum_{k=0}^n q_k \prec
e^a(\lambda_{n+1},+\infty).
\end{align*}
We shall now explain that it easily follows from the preceding estimate that
the projection
\begin{align*}
p_{n+1}=e^a(\lambda_{n+1},+\infty)
\prod_{k=0}^n(\mathbf{1}-p_k)\prod_{k=0}^n(\mathbf{1}-q_k)
\end{align*}
is infinite. Indeed, assume for a moment that the projection
$e^a(\lambda_{n+1},+\infty)[\mathbf{1} - \bigvee_{k=0}^n p_k \vee
\bigvee_{k=0}^n q_k]$ is finite. In this case, $p_{n+1}\prec
e^a(\lambda_{n+1},+\infty)$ and so $e^a(\lambda_{n+1},+\infty)=
e^a(\lambda_{n+1},+\infty)[\bigvee_{k=0}^n p_k \vee
\bigvee_{k=0}^n q_k]+p_{n+1} \prec e^a(\lambda_{n+1},+\infty)$
(Lemma \ref{l0} (ii)). This contradiction shows that the
assumption just made is false.

Now, by Lemma  \ref{l0} (iv), we first deduce
that
\begin{align*}
p_{n+1}\sim e^a(\lambda_{n+1},+\infty),
\end{align*}
next, by the assumption of Lemma \ref{l2}, we have
\begin{align*}e^a(\lambda_{n+1},+\infty)\leq
e^a(\varepsilon\lambda_{n+1},+\infty)\prec
e^a(0,\varepsilon\lambda_{n+1}],
\end{align*}
and finally, again by Lemma  \ref{l0} (iv)
\begin{align*} e^a(0,\varepsilon\lambda_{n+1}]\sim
e^a(0,\varepsilon\lambda_{n+1}]
\prod_{k=0}^n(\mathbf{1}-p_k)\prod_{k=0}^n(\mathbf{1}-q_k).
\end{align*}
Thus,
\begin{align*}
p_{n+1}\prec e^a(0,\varepsilon\lambda_{n+1}]
\prod_{k=0}^n(\mathbf{1}-p_k)\prod_{k=0}^n(\mathbf{1}-q_k).
\end{align*}
and therefore, it follows  from Lemma \ref{l1}(ii), that
there exists
\begin{align*}
P(\mathcal{M})\ni q_{n+1} < e^a(0,\varepsilon\lambda_{n+1}]
\prod_{k=0}^n(\mathbf{1}-p_k)\prod_{k=0}^n(\mathbf{1}-q_k),
\end{align*}
such that $q_{n+1}\sim p_{n+1}$ and $aq_{n+1}=q_{n+1}a$.

Thus the projections $p_{n+1}$ and $q_{n+1}$ are defined and so the construction of the sequences $\{p_n\}_{n=0}^\infty,\ \{q_n\}_{n=0}^\infty$ is also completed for the cases of (b) and (c).

It is clear from the construction that for all these sequences the conditions (i) and (ii) hold. To see that the condition (iii) also holds, we first make the claim  that
\begin{align*}
e^a(-\infty,0]+e^a(\lambda_n,+\infty)\leq \bigvee_{k=0}^n p_k \vee
\bigvee_{k=0}^n q_k,\ n\ge 1.
\end{align*}
To see that the estimate above indeed holds, observe that by the construction, we have
 $e^a(-\infty,0]\leq q_0$ and that by the definition $p_{n+1}:=e^a(\lambda_{n+1},+\infty)[\mathbf{1} -
\bigvee_{k=0}^n p_k \vee \bigvee_{k=0}^n q_k]$. Therefore
$\bigvee_{k=0}^n p_k \vee \bigvee_{k=0}^n q_k \vee p_{n+1} \geq
e^a(\lambda_{n+1},+\infty)$ for all $n\ge 1$ which completes the
justification of the claim above. Now, running $n\to \infty$ we
arrive at the condition (iii).

Now, we can proceed with the construction of the unitary operator
$u_\varepsilon\in\mathcal{M}$ from the assertion.

Let $v_n\in \mathcal{M}$ be a partial isometry such that
$v_n^*v_n=p_n,\ v_nv_n^*=q_n,\ n=0,1,...$. We set
\begin{align*} u_\varepsilon
=\sum_{n=0}^\infty v_n + \sum_{n=0}^\infty v_n^*
\end{align*} (here, the sums are taken in the strong operator topology). Then, we have
\begin{align*} u_\varepsilon^*u_\varepsilon=\sum_{n=0}^\infty p_n +
\sum_{n=0}^\infty q_n = \mathbf{1},\ u_\varepsilon
u_\varepsilon^*=\sum_{n=0}^\infty q_n + \sum_{n=0}^\infty p_n =
\mathbf{1}.\end{align*}
Observe that
\begin{align*}u_\varepsilon p_n = q_n u_\varepsilon,\
u_\varepsilon q_n = p_n u_\varepsilon,\ ap_n=p_na,\ q_na=aq_n,\ n\ge 0,
\end{align*}
and so
the element $u_\varepsilon^* a u_\varepsilon$ commutes
with all the projections $p_n$ and $q_n$, $n\ge 0$. Moreover, since
for all $n\ge 0$, it holds
\begin{align*} ap_n=ae^a(\lambda_n,+\infty)p_n\geq \lambda_n
e^a(\lambda_n,+\infty)p_n =\lambda_n p_n,\\
aq_n=ae^a(-\infty,\varepsilon\lambda_n)q_n\leq
\varepsilon\lambda_n e^a(-\infty,\varepsilon\lambda_n)q_n =
\varepsilon\lambda_n q_n
\end{align*}
we obtain immediately for all such $n$'s
\begin{align*}
u_\varepsilon^* a
u_\varepsilon p_n &=u_\varepsilon^* aq_n u_\varepsilon \leq
\varepsilon\lambda_n u_\varepsilon^* q_n
u_\varepsilon=\varepsilon\lambda_n p_n,\\
u_\varepsilon^* a
u_\varepsilon q_n&=u_\varepsilon^* ap_n u_\varepsilon \geq
\lambda_n u_\varepsilon^* p_n u_\varepsilon=\lambda_n q_n.
\end{align*}
In particular, $(u_\varepsilon^* a u_\varepsilon - a)p_n\leq
\varepsilon\lambda_n p_n-\lambda_n p_n=
-\lambda_n(1-\varepsilon)p_n\leq 0$. Taking into account that
$ap_n\geq \lambda_n p_n$, we now obtain

\begin{align*}
|u_\varepsilon^* a
u_\varepsilon - a|p_n &=(a-u_\varepsilon^* a u_\varepsilon)p_n \geq
ap_n-\varepsilon\lambda_n p_n\\ &\geq ap_n-\varepsilon
ap_n=(1-\varepsilon)ap_n\\ &=(1-\varepsilon)|a|p_n.
\end{align*}
Analogously, for every $n\ge 0$, we have $(u_\varepsilon^* a u_\varepsilon - a)q_n\geq\lambda_n
q_n-\varepsilon\lambda_n q_n= (1-\varepsilon)\lambda_n q_n\geq 0$.
Therefore,
\begin{align*}
|u_\varepsilon^* a u_\varepsilon - a|q_n=(u_\varepsilon^*
a u_\varepsilon - a)q_n& \geq (1-\varepsilon)\lambda_n q_n \\ &\geq
(1-\varepsilon)a q_n.
\end{align*}
Observe that the inequalities above hold for all $n\ge 0$. If
$n>0$, then $q_n<e^a(0,\varepsilon\lambda_n]$, $q_na=aq_n$ by the
construction and so $aq_n=|a|q_n$, that is we have
$$
|u_\varepsilon^* a u_\varepsilon - a|q_n\geq (1-\varepsilon)|a|q_n.
$$

A little bit more care is required when $n=0$. In this case,
recall that $q_0=e^a(-\infty,0]+q_0^1$, where
$q_0^1<e^a(0,\varepsilon\lambda_0]$. Obviously,
$ae^a(-\infty,0]\leq 0$, and so
$ae^a(-\infty,0]=-|a|e^a(-\infty,0]$. Therefore since (see above)
$u_\varepsilon^* a u_\varepsilon q_0\geq \lambda_0 q_0$ and
$aq_0=ae^a(-\infty,0]+aq_0^1=-|a|e^a(-\infty,0]+aq_0^1$, we have
\begin{align*}
|u_\varepsilon^* a u_\varepsilon - a|q_0 & \geq (u_\varepsilon^* a
u_\varepsilon - a)q_0\geq  \lambda_0
q_0-aq_0^1+|a|e^a(-\infty,0]\\ & \geq
\lambda_0q_0^1-\varepsilon\lambda_0
q_0^1+|a|e^a(-\infty,0]=(1-\varepsilon)\lambda_0
q_0^1+|a|e^a(-\infty,0]\\  & \geq (1-\varepsilon)a
q_0^1+|a|e^a(-\infty,0]=
(1-\varepsilon)|a|q_0^1+|a|e^a(-\infty,0]\\ &\geq
(1-\varepsilon)(|a|q_0^1+|a|e^a(-\infty,0])=(1-\varepsilon)|a|q_0.
\end{align*}

Collecting all preceding inequalities, we see that for every $k\ge 0$ we have

\begin{align*}
|u_\varepsilon^* a u_\varepsilon - a|\sum_{n=0}^k (p_n+q_n)\geq (1-\varepsilon)|a|\sum_{n=0}^k (p_n+q_n)
\end{align*}
and since $\sum_{n=0}^\infty (p_n+q_n)=\mathbf{1}$, we conclude
\begin{align*}
|u_\varepsilon^* a u_\varepsilon - a|\geq (1-\varepsilon)|a|.
\end{align*}

The assertion of the lemma now follows by observing that %
%
%
$|u_\varepsilon^* a u_\varepsilon -
a|=|[a,u_\varepsilon]|$.
\end{proof}

The following lemma is somewhat similar to \cite[Proposition 5.6]{BdPS} proved there for $II_1$-factors. We however need its modification (and strengthening) for general $W^*$-factors.

\begin{lemma}
\label{l3} Suppose that there exists $\lambda\in\mathbb{R}$ and
projections $p,q\in P(\mathcal{M})$, such that $p,q\leq
e^a\{\lambda\},\ pq=0$ and $e^a(-\infty,\lambda)+p\sim
e^a(\lambda,+\infty)+q$. Then there exists an element  $u=u^*\in
U(\mathcal{M})$, satisfying \eqref{first main}.
\end{lemma}
\begin{proof}
Set $r:=\mathbf{1}-(e^a(-\infty,\lambda)+p+
e^a(\lambda,+\infty)+q)$. Then $p,q,r\leq e^a\{\lambda\}$ and so
 $ap=\lambda p,\ aq=\lambda q,\ ar=\lambda r$. We claim that there exists a self-adjoint unitary element $u$ such that $u(e^a(-\infty,\lambda)+p)=(e^a(\lambda,+\infty)+q)u,\ ur=r$. Indeed, since $e^a(-\infty,\lambda)+p\sim
e^a(\lambda,+\infty)+q$, there exists a partial isometry  $v$ such that
$v^*v=e^a(-\infty,\lambda)+p,\ vv^*=e^a(\lambda,+\infty)+q$. Set $u:=v+v^*+r$. We have
$u^*u=e^a(-\infty,\lambda)+p+e^a(\lambda,+\infty)+q+r=\mathbf{1}$,
$uu^*=e^a(\lambda,+\infty)+q+e^a(-\infty,\lambda)+p+r=\mathbf{1}$,
$u^*=v^*+v+r=u$. This establishes the claim. It now remains to verify that \eqref{first main} holds.

To this end, first of all observe that the operators $a$ and $u^*au$ commute with the
projections $e^a(-\infty,\lambda)+p$, $e^a(\lambda,+\infty)+q$ and
$r$. This observation guarantees that
\begin{align*}
(u^*au-a)(e^a(-\infty,\lambda)+p)&=|u^*au-a|(e^a(-\infty,\lambda)+p),\\
(u^*au-a)(e^a(\lambda,+\infty)+q)&=|u^*au-a|(e^a(\lambda,+\infty)+q)
\end{align*}
and so
\begin{align*}
|u^*au-a|(e^a(-\infty,\lambda)+p)&=
u^*a(e^a(\lambda,+\infty)+q)u-a(e^a(-\infty,\lambda)+p)\\
&=u^*a(e^a(\lambda,+\infty)+q)u-\lambda
u^*(e^a(\lambda,+\infty)+q)u\\
&+
\lambda(e^a(-\infty,\lambda)+p)-a(e^a(-\infty,\lambda)+p)\\
&=u^*|a(e^a(\lambda,+\infty)+q)-\lambda(e^a(\lambda,+\infty)+q)|u\\
&+|\lambda(e^a(-\infty,\lambda)+p)-a(e^a(-\infty,\lambda)+p)|\\
&=u^*|a-\lambda\mathbf{1}|u(e^a(-\infty,\lambda)+p)+
|a-\lambda\mathbf{1}|(e^a(-\infty,\lambda)+p).
\end{align*}
and similarly
\begin{align*}
|u^*au-a|(e^a(\lambda,+\infty)+q)&=
u^*a(e^a(-\infty,\lambda)+p)u-a(e^a(\lambda,+\infty)+q)\\
&=u^*a(e^a(-\infty,\lambda)+p)u-\lambda
u^*(e^a(-\infty,\lambda)+p)u\\
&+ \lambda (e^a(\lambda,+\infty)+q)
-a(e^a(\lambda,+\infty)+q)\\
&=-u^*|a-\lambda\mathbf{1}|u(e^a(\lambda,+\infty)+q)-
|a-\lambda\mathbf{1}|(e^a(\lambda,+\infty)+q).
\end{align*}
Finally, $(u^*au-a)r=\lambda r -\lambda r=0$, that is,
$|u^*au-a|r=0$.
We now obtain \eqref{first main} as follows

\begin{align*}|u^*au-a|&=|u^*au-a|[(e^a(-\infty,\lambda)+p)+(e^a(\lambda,+\infty)+q)+r]\\ & =
|u^*au-a|(e^a(-\infty,\lambda)+p)+|u^*au-a|(e^a(\lambda,+\infty)+q)+|u^*au-a|r\\ & =
(u^*|a-\lambda\mathbf{1}|u+|a-\lambda\mathbf{1}|)[(e^a(-\infty,\lambda)+p)+(e^a(\lambda,+\infty)+q)+r]\\ & =
u^*|a-\lambda\mathbf{1}|u+|a-\lambda\mathbf{1}|.
\end{align*}\end{proof}

The following lemma is well known. We include a short proof for convenience.

\begin{lemma}
\label{l4}  Let $I$ be an arbitrary ideal in an arbitrary
$W^*$-algebra $\mathcal{A}$. Then $x\in I\Leftrightarrow |x|\in
I\Leftrightarrow x^*\in I$. Furthermore, if $0\leq x\leq y\in I$,
then $x\in I$.
\end{lemma}
\begin{proof}
If $x\in I$ and $x=v|x|$ is polar decomposition of $x$, then
$|x|=v^*x \in I$ and $x^*=|x|v^*\in I$.

Let $x,y\in\mathcal{A},\ 0\leq x\leq y\in I$. In this case there
exists an element $z\in\mathcal{A}$, such that $x^{1/2}=zy^{1/2}$
\cite[Ch.11, Lemma 2]{Dix}. Then $x^{1/2}=(x^{1/2})^*=y^{1/2}z^*$
and $x=x^{1/2}x^{1/2}=zy^{1/2}y^{1/2}z^*=zyz^*\in I$.
\end{proof}

We are now fully equipped to prove Theorem \ref{main}.

\begin{proof}[Proof of Theorem \ref{main}]
We concentrate first at proving assertions (i) and (ii) of of Theorem \ref{main}.
Let us consider the splitting of the set $\mathbb{R}$ of all real numbers into the following pairwise disjoint subsets:
\begin{align*}
\Lambda_- &:=\{\lambda\in\mathbb{R}:\ e^a(-\infty,\lambda)\prec
e^a(\lambda,+\infty)\},\\
\Lambda_0&:=\{\lambda\in\mathbb{R}:\
e^a(-\infty,\lambda)\sim e^a(\lambda,+\infty)\},\\
\Lambda_+&:=\{\lambda\in\mathbb{R}:\ e^a(-\infty,\lambda)\succ
e^a(\lambda,+\infty)\}.
\end{align*}

If $\Lambda_0\neq\emptyset$, then the assumptions of Lemma
\ref{l3} hold for all $\lambda\in\Lambda_0$. Thus in this case for
$a$, the assertion (i) of Theorem \ref{main} follows immediately from that lemma and hence the
assertion (ii) of that Theorem trivially holds as well.

In the rest of the proof, we shall assume that $\Lambda_0=\emptyset$.

Note that if $\lambda\in\Lambda_-$ and $\mu<\lambda$, then
\begin{align*}
e^a(-\infty,\mu)\leq e^a(-\infty,\lambda)\prec
e^a(\lambda,+\infty)\leq e^a(\mu,+\infty),
\end{align*}
that is, $\mu\in\Lambda_-$. The analogous assertion for
$\Lambda_+$ is proved similarly. These observations immediately
imply that $\Lambda_-$ and $\Lambda_+$ are connected subsets in
$\mathbb{R}$ and so for all  $\lambda_-\in\Lambda_-$ and
$\lambda_+\in\Lambda_+$, we have $\lambda_-<\lambda_+$. We shall now show that both sets $\Lambda_-$
and $\Lambda_+$ are nonempty. Suppose for a moment that $\Lambda_-= \emptyset$.
Since  $a\in S_h(\mathcal{M})$ there exists some $\lambda_1>0$, such that all projections $e^a(-\infty,\mu)$ for $\mu<-\lambda_1$ and
$e^a(\mu,+\infty)$ for $\mu>\lambda_1$ are finite, and
$e^a(-\infty,\mu) \to 0$ as $\mu\to -\infty$ and $e^a(\mu,\infty)
\to 0$ as $\mu\to\infty$. Let
$\lambda_n\downarrow -\infty,\ \lambda_1=-\mu$. By the assumption
 $\lambda_n\notin \Lambda_-$ for all $n\ge 1$. Fixing $n$ and tending $k$ to infinity
we have $e^a(-\infty,\lambda_{n+k})\succeq
e^a(\lambda_{n+k},+\infty)\geq e^a(\lambda_n,+\infty)$. However, all
projections $e^a(-\infty,\lambda_{n+k})$ are finite and
$e^a(-\infty,\lambda_{n+k})\downarrow 0$, therefore
$e^a(\lambda_n,+\infty)=0$ for any $n\in\mathbb{N}$ (see \cite[Lemma 6.11]{BdPS},
). On the other hand,
$e^a(\lambda_n,+\infty)\uparrow\mathbf{1}$. This contradiction shows that
$\Lambda_-\neq \emptyset$. The assertion  $\Lambda_+\neq
\emptyset$ is established with a similar argument.

Therefore, there exists such a unique $\lambda_0\in\mathbb{R}$, satisfying $(-\infty,\lambda_0)\subset
\Lambda_-$ and $(\lambda_0,+\infty)\subset\Lambda_+$.

Consider the case when both projections $e^a(-\infty,\lambda_0)$
and $e^a(\lambda_0,+\infty)$ are finite. Since
$\Lambda_0=\emptyset$, we have that these two projections are not
pairwise equivalent. For definiteness, let us assume
$e^a(-\infty,\lambda_0)\prec e^a(\lambda_0,+\infty)$ ( the case
when $e^a(\lambda_0,+\infty)\prec e^a(-\infty,\lambda_0)$ is
treated similarly). Then there exists $r\in P(\mathcal{M})$, such
that $e^a(-\infty,\lambda_0)\sim r < e^a(\lambda_0,+\infty)$. If
$\mathcal{M}$ is an infinite factor, then $e^a\{\lambda_0\}$ is an
infinite projection. Therefore, there exists $p\in
P(\mathcal{M})$, such that $e^a(\lambda_0,+\infty)-r\sim
p<e^a\{\lambda_0\}$. Then the pair $(a,\lambda_0)$ satisfies the
assumption of Lemma \ref{l3}. Indeed, setting $q=0$, we have
$e^a(-\infty,\lambda_0)+p\sim
r+(e^a(\lambda_0,+\infty)-r)=e^a(\lambda_0,+\infty)$, $q=0$. As above this yields the assertions (i) and (ii) of Theorem \ref{main}.

If $\mathcal{M}$ is a finite factor, then there exists a faithful
normal trace $\tau$ on $\mathcal{M}$ such that
$\tau(\mathbf{1})=1$ (that is, $\tau$ is normalized)
\cite[\S8.5]{KR2}. Certainly, we have
$\tau(e^a(-\infty,\lambda))\leq 1/2$ for all $\lambda\in\Lambda_-$
and $\tau(e^a(\lambda,+\infty))\leq 1/2$ for all
$\lambda\in\Lambda_+$. Therefore, by the normality of $\tau$ it
follows that $\tau(e^a(-\infty,\lambda_0))\leq 1/2$ and
$\tau(e^a(\lambda_0,+\infty))\leq 1/2$. Thus if we have
$e^a(-\infty,\lambda_0)\preceq e^a(\lambda_0,+\infty)$, then
 there exists a projection $p\leq e^a\{\lambda_0\}$ such that
$e^a(-\infty,\lambda_0)+p\sim e^a(\lambda_0,+\infty)$. Hence in
this case, both projections $e^a(-\infty,\lambda_0)$ and
$e^a(\lambda_0,+\infty)$ are finite, and (setting $q=0$ as above)
we see that the assumption of Lemma \ref{l3} holds.
We note, in passing, that a similar argument occurred also in
\cite[Corollary 2.7]{Hl}. So, in this case again the assertions
(i) and (ii) of Theorem \ref{main} hold.

Note that by now we have completed the proof of Theorem \ref{main} (i), (ii) for the case when
$\mathcal{M}$ is a finite factor. Moreover, we have also finished
the proof for the case when $\mathcal{M}$ is an infinite factor
and both projections $e^a(-\infty,\lambda_0)$ and
$e^a(\lambda_0,+\infty)$ are finite.

Let us now consider the case when $\mathcal{M}$ is
a purely infinite $\sigma$-finite factor. In such a factor, all
nonzero projections are infinite and are equivalent to each other
\cite[Proposition $V$.1.39]{Tak}. Therefore, in this case, we may
assume that both projections  $e^a(-\infty,\lambda_0)$ and
$e^a(\lambda_0,+\infty)$ are infinite or otherwise one of these
projections must be $0$. Our strategy is to show that in this case
$\Lambda_0\neq\emptyset$. This would yield a contradiction with
the assumption $\Lambda_0=\emptyset$ made earlier and would
complete the proof.

Suppose that $e^a(-\infty,\lambda_0)\prec e^a(\lambda_0,+\infty)$
 that is assume that $\lambda_0\in \Lambda_-$. Then
$e^a(-\infty,\lambda_0)=0$ since all nonzero projections in
$\mathcal{M}$ are equivalent. Furthermore, since for all
$\lambda>\lambda_0$ we have $\lambda\in\Lambda_+$, a similar
argument yields $e^a(\lambda,+\infty)=0$ for all such $\lambda$'s.
Thus
\begin{align*}
e^a(\lambda_0,+\infty)=
\bigvee_{\lambda>\lambda_0} e^a(\lambda,+\infty)=0,
\end{align*}
and we obtain  $0=e^a(-\infty,\lambda_0)\sim
e^a(\lambda,+\infty)=0$ that is $\lambda_0\in\Lambda_0$. However,
this contradicts our assumption that $\Lambda_0 = \emptyset$. The
case when $e^a(-\infty,\lambda_0) \succ e^a(\lambda_0,+\infty)$ is
considered analogously.  This completes the proof of
assertions (i) and (ii) of Theorem \ref{main} for the case of a
purely infinite $\sigma$-finite factor $\mathcal{M}$.

To finish the proof of the assertion (ii), it
remains to consider the case of an infinite factor $\mathcal{M}$,
that is when $\mathcal{M}$ is of type $II_\infty$ or else when
$\mathcal{M}$ is of type $I_\infty$, or else when $\mathcal{M}$ is
a non-$\sigma$-finite factor of type $III$ and when at least one
of the projections $e^a(-\infty,\lambda_0)$ and
$e^a(\lambda_0,+\infty)$ is properly infinite.

In fact, it is sufficient to consider only the case when
\begin{equation}\label{the last case}e^a(-\infty,\lambda_0)\prec
e^a(\lambda_0,+\infty)
\end{equation} that is, $\lambda_0\in\Lambda_-$. Indeed, if the
assertion (ii) of Theorem \ref{main} holds under the assumption
\eqref{the last case}, then the remaining case when:
$e^a(-\infty,\lambda_0)\succ e^a(\lambda_0,+\infty)$ (that is,
$\lambda_0\in\Lambda_+$) is reduced to \eqref{the last case} by
substituting $a$ for $-a$ and $\lambda_0$ for $-\lambda_0$.

Assume now that \eqref{the last case} holds (in this case,
 the projection $e^a(\lambda_0,+\infty)$ is necessarily infinite).

We may also further assume that the assumptions of Lemma \ref{l3} do
not hold (otherwise, there is nothing to prove).

We shall now show that  \begin{equation}\label{more of the last case}
e^a(-\infty,\lambda_0] \prec
e^a(\lambda_0,+\infty).\end{equation}
Suppose the contrary, that is that either
\begin{equation}\label{contrary} e^a(-\infty,\lambda_0] = e^a(\lambda_0,+\infty) +e^a\{\lambda_0\}
\succ e^a(\lambda_0,+\infty)\end{equation}
or else that
\begin{equation}\label{contrary-2}
e^a(-\infty,\lambda_0] \sim e^a(\lambda_0,+\infty).
\end{equation}
If \eqref{contrary-2} holds then setting $p=e^a\{\lambda_0\},\ q=0$, we have
$e^a(-\infty,\lambda_0]=e^a(-\infty,\lambda_0)+p$ we arrive at the setting when the
assumptions of Lemma \ref{l3} hold and we are done. Suppose now that \eqref{contrary} holds.  Then by Lemma \ref{l00}, it follows
from \eqref{more of the last case} that there exists a projection
$p\in P(\mathcal{M})$, for which $e^a(\lambda_0,+\infty)\sim
e^a(-\infty,\lambda_0)+p$ and $p<e^a\{\lambda_0\}$. However, this
again means that the assumptions of Lemma \ref{l3} hold (with $q=0$). This completes the proof of
\eqref{more of the last case}.

Our next claim is that
\begin{equation}\label{yet more of the last
case} e^a(\lambda_0,\lambda] \succ
e^a(-\infty,\lambda_0]+e^a(\lambda,+\infty)\end{equation} for all
$\lambda>\lambda_0$. Suppose the contrary
\begin{equation}\label{another contrary}
e^a(\lambda_0,\lambda] \preceq
e^a(-\infty,\lambda_0]+e^a(\lambda,+\infty)
\end{equation}
for some $\lambda>\lambda_0$. Setting for a moment $p:=e^a(-\infty,\lambda_0]+e^a(\lambda,+\infty),\
\mathbf{1}-p=e^a(\lambda_0,\lambda]$, we rewrite \eqref{another contrary} as $p\succeq \mathbf{1}-p$, and, thanks to
 Lemma \ref{l0}(iii), conclude that
$$
e^a(-\infty,\lambda_0]+e^a(\lambda,+\infty) \sim \mathbf{1}.
$$
However, \eqref{more of the last case} implies
$$
e^a(-\infty,\lambda_0] \prec \mathbf{1},
$$
and, by Lemma \ref{l0}(ii), we obtain $e^a(\lambda,+\infty) \sim \mathbf{1}$. However, for $\lambda>\lambda_0$ we have
$\lambda\in\Lambda_+$, and so $e^a(-\infty,\lambda) \succ
e^a(\lambda,+\infty) \sim \mathbf{1}$, that is  $e^a(-\infty,\lambda)
\succ \mathbf{1}$ which is obviously impossible. Hence, \eqref{another contrary} fails and \eqref{yet more of the last case} holds.

Now, observe that the condition \eqref{yet more of the last case} means that the element $a-\lambda_0 \mathbf{1}$ satisfied the assumptions of
Lemma \ref{l2}. The assertion (ii) of Theorem \ref{main}
now follows from that lemma. This completes the proof of the assertions (i) and (ii).

Let us prove the final assertion of Theorem
\ref{main}. To this end, let $\mathcal{M}$ be an infinite
semi-finite $\sigma$-finite factor. Fix a sequence
$\{p_n\}_{n\ge 1}$ of pairwise disjoint finite projections
$p_1,p_2,...,p_n,...$ such that $\bigvee_{n=1}^\infty p_n
=\mathbf{1}$ (any maximal family of  pairwise disjoint finite projections
in $\mathcal{M}$ is countable)
and set
$$
a:=\sum_{n=1}^\infty n^{-1}p_n.
$$
We have $a=a^*\in \mathcal{M}\cap \mathcal{F}$, where $\mathcal{F}$ is
the norm-closed ideal, generated by the elements $x\in\mathcal{F}$ such that $r(x)$
(and hence $l(x)$) is a finite projection in $\mathcal{M}$.
Moreover, the support of $a$, $s(a)$, is equal to $\mathbf{1}$.
Suppose that
\begin{equation}\label{last effort}
|[a,u]|\geq |a-\lambda\mathbf{1}|
\end{equation} for some
$\lambda\in\mathbb{C}$ and some $u\in U(\mathcal{M})$. Since $[a,u]\in \mathcal{F}$, we have by Lemma \ref{l4} that also
$a-\lambda\mathbf{1} \in \mathcal{F}$. However, the set
$\{a-\lambda\mathbf{1}:\ \lambda \in \mathbb{C}\}$ may contain at most one element can belong to $\mathcal{F}$, since $\mathcal{F}$ is a proper ideal in
$\mathcal{M}$ (see \cite[Theorem 6.8.7]{KR2}). This guarantees that $\lambda=0$, that is
$|u^*au-a|=|[a,u]|\geq |a|=a$. Let $e_+:=s((u^*au-a)_+),\
e_-:=\mathbf{1}-e_+$. We have $e_-(a-u^*au)e_-\geq e_-ae_-$, or equivalently,
$e_-u^*aue_-\leq 0$. Since $u^*au\ge 0$, we conclude $e_-u^*aue_-=0$, or equivalently,
$a^{1/2}ue_-=0$, which in turn implies $aue_-u^*=0$. Thus,
$$
a=au(e_++e_-)u^*=aue_+u^*,
$$
in particular $ue_+u^*\geq s(a)=\mathbf{1}$ and therefore $e_+=\mathbf{1}$. On the other
hand, due to the definition of $e_+$ and the inequality $|u^*au-a|\geq a$, we have $e_+(u^*au-a)e_+\geq e_+ae_+$, or equivalently, $u^*au\geq 2a$. However, the preceding inequality implies that $1=\|a\|\geq 2\|a\|=2$ and therefore is false. This contradiction shows that
$\lambda$ and $u$ satisfying \eqref{last effort} do not exist.
\end{proof}

\section{Applications of Theorem \ref{main} to derivations}

Recall that a \textit{derivation} on a complex algebra $A$ is a linear map $\delta :A\rightarrow A$ such that
\begin{equation*}
\delta \left( xy\right) =\delta \left( x\right) y+x\delta \left(
y\right) ,\ \ \ x,y\in A.
\end{equation*}
If $a\in A$, then the map $\delta _{a}:A\rightarrow A$, given by
$\delta _{a}\left( x\right) =\left[ a,x\right] $, $x\in A$, is a
derivation. A derivation of this form is called \textit{inner}.

Our first result here is somewhat similar (at least in spirit) to some results in \cite{JP, KW, PR}.

\begin{corollary}
\label{c1} Let $\mathcal{M}$ be a $W^*$ -factor and $I$ be an ideal
in $\mathcal{M}$ and let $\delta: \mathcal{M} \rightarrow I$ be a
derivation. Then there exists an element $a\in I$, such that
$\delta=\delta_a=[a,.]$.
\end{corollary}
\begin{proof}
Since $\delta$ is a derivation on a $W^*$-algebra, it is necessarily inner \cite[Theorem 4.1.6]{Sak}. Thus there exists an element $d\in\mathcal{M}$, such that
$\delta(\cdot)=\delta_d(\cdot)=[d,.]$. It follows from our hypothesis that
$[d,\mathcal{M}]\subseteq I$.

Using Lemma \ref{l4}, we obtain $[d^*,\mathcal{M}]=-[d,\mathcal{M}]^*\subseteq I^*=I$
and $[d_k,\mathcal{M}]\subseteq I,\ k=1,2$, where $d=d_1+i d_2$,
$d_k=d_k^*\in \mathcal{M}$, for $k=1,2$. It follows now from Theorem
\ref{main},  that there exist scalars $\lambda_1,\lambda_2\in\mathbb{R}$ and $u_1,u_2\in
U(\mathcal{M})$, such that $|[d_k,u_k]|\geq
1/2|d_k-\lambda_k\mathbf{1}|$ for $k=1,2$. Again applying Lemma \ref{l4}, we obtain
$d_k-\lambda_k\mathbf{1}\in I$, for $k=1,2$.
Setting $a:=(d_1-\lambda_1\mathbf{1})+i(d_2-\lambda_2\mathbf{1})$,
we deduce that $a\in I$ and $\delta=[a,.]$.
\end{proof}

Classical examples of ideals $I$ satisfying the
assumptions of Corollary \ref{c1} above are given by symmetric
operator ideals.

\begin{definition}\label{opideal}
If $\mathcal I$ is a $*$-ideal in a von Neumann algebra $\mathcal N$
which is complete in a norm $\|\cdot\|_{\mathcal I}$
then we will call $\mathcal I$ a symmetric
operator ideal if\\
(1) $\| S\|_{\mathcal I}\geq \| S\|$ for all $S\in \mathcal I$,\\
(2) $\| S^*\|_{\mathcal I} = \| S\|_{\mathcal I}$ for all $S\in \mathcal I$,\\
(3) $\| ASB\|_{\mathcal I}\leq \| A\| \:\| S\|_{\mathcal I}\| B\|$ for all
$S\in \mathcal I$, $A,B\in \mathcal N$.\\
Since $\mathcal I$ is an ideal in a von Neumann algebra, it follows
from I.1.6, Proposition 10 of \cite{Dix} that if $0\leq S \leq T$
and $T \in {\mathcal I}$, then $S \in {\mathcal I}$ and $\| S\|_{\mathcal I}
\leq \| T\|_{\mathcal I}$.
\end{definition}

\begin{corollary}
\label{c11} Let $\mathcal{M}$ be a $W^*$ -factor and $I$ be a
symmetric operator ideal in $\mathcal{M}$ and let $\delta:
\mathcal{M} \rightarrow I$ be a self-adjoint derivation. Then
there exists an element $a\in I$, such that
$\delta=\delta_a=[a,.]$ and that
$\|a\|_I\leq\|\delta\|_{\mathcal{M}\rightarrow I}$.
\end{corollary}
\begin{proof} Firstly, we observe that
$\|\delta\|_{\mathcal{M}\rightarrow I}<\infty$. Indeed, we have
$\delta=\delta_a,\ a\in I$ and therefore $\|\delta(x)\|_I=\|ax-xa\|_I\leq
\|ax\|_I+\|xa\|_I\leq 2\|a\|_I \|x\|_{\mathcal{M}}$, that is
$\|\delta\|_{\mathcal{M}\rightarrow I}\leq 2\|a\|_I<\infty$.

Let now $\delta$ be a self-adjoint derivation on $\mathcal{M}$,
that is $\delta(\cdot)=\delta_d(\cdot)=[d,\cdot]$ for some $d\in
\mathcal{M}$, such that $[d,x]^*=[d,x^*]$ for all
$x\in\mathcal{M}$. We have $x^*d^*-d^*x^*=dx^*-x^*d$, that is,
$x^*(d^*+d)=(d^*+d)x^*$ for all $x\in\mathcal{M}$. This
immediately implies $\mathfrak{Re}(d)\in Z(\mathcal{M})$ and so,
we can safely assume that $\delta=\delta_{id}=[id,.]$, where $d$
is a self-adjoint operator from $\mathcal{M}$. Fix $\varepsilon>0$
and let $\lambda_0\in\mathbb{R}$, $u_\varepsilon \in
U(\mathcal{M})$ be such that
$$|[d,u_\varepsilon]| \geq (1-\varepsilon)|d-\lambda_0\mathbf{1}|.
$$ The assumption on $(I,\|\cdot\|)$ guarantees that
$(1-\varepsilon)\|d-\lambda_0\mathbf{1}\|_I\leq
\|\delta(u_\varepsilon)\|_I\leq \|\delta\|_{\mathcal{M}\rightarrow
I}$. Since $\varepsilon$ was chosen arbitrarily, we conclude that
$\|d-\lambda_0\mathbf{1}\|_I\leq
\|\delta\|_{\mathcal{M}\rightarrow I}$. Setting
$a=i(d-\lambda_0\mathbf{1})$ completes the proof.
\end{proof}

If the von Neumann algebra $\mathcal N$ is equipped with a
faithful normal semi-finite trace $\tau$, then the set
$$\mathcal{L}_p(\mathcal{N})=\left\{S\in {\mathcal{N}}:\ \tau(|S|^p)<\infty\right\}$$
equipped  with a standard norm
$$\|S\|_{\mathcal{L}_p(\mathcal{N})}=\max\{\|S\|_{B(H)},\ \tau(|S|^p)^{1/p}\}$$
is called Schatten-von Neumann $p$-class. In the Type $I$ setting these are the usual Schatten-von Neumann ideals. The result of Corollary \ref{c1} complements results given in \cite[Section 6]{KW}.

A closely linked example is the following. Consider the ideal $\mathcal{K}_{\mathcal {N}}$ of $\tau$-compact operators in $\mathcal {N}$ (that is
the norm closed ideal generated by the projections $E\in P(\mathcal {N})$ with $\tau(E) < \infty$). In this special case, the result of Corollary \ref{c11} is analogous to the classical result that any derivation on $B(H)$ taking values in the ideal of compact operators on $H$
can be represented as $\delta_a$ with $a$ being a compact operator (see e.g. \cite[Lemma 3.2]{JP}).

We now consider analogues of Corollary \ref{c1} for ideals of (unbounded) $\tau$-measurable operators.

\begin{corollary}
\label{c2} Let $\mathcal{M}$ be a $W^*$ -factor and let
$\mathcal{A}$ be a linear subspace in $S(\mathcal{M})$, such that
$\mathcal{A}^*=\mathcal{A},\ x\in\mathcal{A} \Leftrightarrow
|x|\in\mathcal{A},\ 0<x<y\in\mathcal{A}\Rightarrow
x\in\mathcal{A}$. Fix $a\in S(\mathcal{M})$ and consider inner
derivation $\delta=\delta_a$ on the algebra $S(\mathcal{M})$ given
by $\delta(x)=[a,x]$, $x\in S(\mathcal{M})$. If
$\delta(\mathcal{M})\subseteq\mathcal{A}$, then there exists
$d\in\mathcal{A}$ such that $\delta(x)=[d,x]$.

\end{corollary}
\begin{proof}
Let $a=a_1+i a_2$, where $a_1=\mathfrak{Re}(a)$ and $a_2=\mathfrak{Im}(a)$. We have
$2[a_1,x]=[a+a^*,x]=[a,x]-[a,x^*]^*=\mathcal{A}-\mathcal{A}^*\subseteq\mathcal{A}$
for any $x\in\mathcal{M}$. Analogously, $[a_2,x]\in\mathcal{A}$
for any $x\in\mathcal{M}$. By Theorem \ref{main}, there is a
scalar $\lambda_k\in\mathbb{R}$ and a unitary element
$u_k\in U(\mathcal{M})$, such that $|[a_k,u_k]|\geq
1/2|a_k-\lambda_k\mathbf{1}|$ for $k=1,2$. The assumption on $\mathcal{A}$ guarantees that
$a_k-\lambda_k\mathbf{1}\in\mathcal{A}$, for $k=1,2$. Setting
$d=(a_1-\lambda_1\mathbf{1})+i(a_2-\lambda_2\mathbf{1})$, we
deduce that $d\in\mathcal{A}$ and $\delta=[d,\cdot]$.
\end{proof}

Numerous examples of absolutely solid
subspaces $\mathcal{A}$ in $S(\mathcal{M},\tau)$ satisfying the
assumptions of the preceding corollary are given by $\mathcal{M}$-bimodules of $S\left(\mathcal{M}, \tau \right)$.

\begin{definition}
A linear subspace $E$ of $S\left(\mathcal{M}, \tau \right)$, is called an $\mathcal{M}$
-bimodule of $\tau $-measurable operators if $uxv\in E$
whenever $x\in E$ and $u,v\in \mathcal{M}$. If an $\mathcal{M}$-bimodule $E$
is equipped with a (semi-) norm $\left\Vert \cdot \right\Vert _{E}$,
satisfying
\begin{equation}
\left\Vert uxv\right\Vert _{E}\leq \left\Vert u\right\Vert _{B\left(
H\right) }\left\Vert v\right\Vert _{B\left( H\right) }\left\Vert
x\right\Vert _{E},\ \ \ x\in E,\ u,v\in \mathcal{M}\text{,}  \label{ChIVeq21}
\end{equation}%
then $E$ is called a (semi-) normed $\mathcal{M}$-bimodule of $\tau $
-measurable operators.
\end{definition}

We omit a straightforward verification of the fact that every $\mathcal{M}$-bimodule
of $\tau $-measurable operators satisfies the assumption of Corollary \ref{c2}.

The best known examples of normed $\mathcal{M}$-bimodules of $S\left(\mathcal{M}, \tau \right)$ are given by
the so-called symmetric operator spaces (see e.g. \cite{DDP0, SC, KS}). We briefly recall relevant definitions.



Let $L_0$ be a space of Lebesgue measurable functions either on $(0,1)$ or on $(0,\infty)$, or on $\mathbb{N}$ finite almost everywhere (with identification $m-$a.e.). Here $m$ is Lebesgue measure or else counting measure on $\mathbb{N}$. Define $S_0$ as the subset of $L_0$ which consists of all functions $x$ such that $m(\{|x|>s\})$ is finite for some $s.$

Let $E$  be a Banach space of real-valued Lebesgue measurable functions either on $(0,1)$ or $(0,\infty)$ (with identification $m-$a.e.). $E$ is said to be {\it ideal lattice} if $x\in E$ and $|y|\leq |x|$ implies that $y\in E$ and $||y||_E\leq||x||_E.$

The ideal lattice $E\subseteq S_0$ is said to be {\it symmetric space} if for every $x\in E$ and every $y$ the assumption $y^*=x^*$ implies that $y\in E$ and $||y||_E=||x||_E.$

Here, $x^*$ denotes the non-increasing right-continuous rearrangement of $x$ given by
$$x^*(t)=\inf\{s\geq0:\ m(\{|x|\geq s\})\leq t\}.$$

If $E=E(0,1)$ is a symmetric space on $(0,1),$ then
$$L_{\infty}\subseteq E\subseteq L_1.$$

If $E=E(0,\infty)$ is a symmetric space on $(0,\infty),$ then
$$L_1\cap L_{\infty}\subseteq E\subseteq L_1+L_{\infty}.$$

Let a semi-finite von Neumann  algebra $\mathcal N$ be equipped with a  faithful normal semi-finite trace $\tau$. Let $x\in S(\mathcal{N},\tau)$.
The generalized singular value
function of $x$ is $\mu (x):t\rightarrow \mu _{t}(x),$ where, for
$0\leq t<\tau (\bf {1})$

\begin{equation*}
\mu _{t}(x)=\inf \{s\geq 0\mid \tau (e^{\left\vert x\right\vert
}(s,\infty )\leq t\}.
\end{equation*}

Let $\mathcal{E}$ be a linear subset in $S({\mathcal{N}, \tau})$ equipped with a complete norm $\|\cdot\|_{\mathcal{E}}$. We say that $\mathcal{E}$ is a \textit{symmetric operator space} (on $\mathcal{N}$) if $x\in E$ and every $y\in S({\mathcal{N}, \tau})$ the assumption $\mu(y)\leq \mu(x)$ implies that $y\in E$ and $\|y\|_\mathcal{E}\leq \|x\|_\mathcal{E}$. The fact that every symmetric operator space $\mathcal{E}$ is (an absolutely solid) $\mathcal{M}$-bimodule of $S\left(\mathcal{M}, \tau \right)$ is well known (see e.g. \cite{SC} and references therein).

There exists a strong connection between symmetric function and operator spaces.

Let $E$ be a symmetric function space on the interval $(0,1)$ (respectively, on the semi-axis or on $\mathbb{N}$) and let $\mathcal{N}$ be a type $II_1$ (respectively, $II_{\infty}$ or type $I$) von Neumann algebra. Define
$$E(\mathcal{N},\tau):=\{S\in S(\mathcal{N},\tau):\ \mu_t(S)\in E\},\ \|S\|_{E(\mathcal{N},\tau)}:=\|\mu_t(S)\|_E.$$
Main results of \cite{KS} assert that $(E(\mathcal{N},\tau), \|\cdot\|_{E(\mathcal{N},\tau)})$ is a symmetric operator
space. If $E=L_p$, $1\leq p<\infty$, then $(E(\mathcal{N},\tau), \|\cdot\|_{E(\mathcal{N},\tau)})$ coincides with the classical noncommutative $L_p$-space associated with the algebra $(\mathcal{N},\tau)$. If $\mathcal{N}$ is a semi-finite  factor, then the converse result is trivially true. That is assume for definiteness that $\mathcal{N}$ is $II_\infty$-factor and that $\mathcal{E}$ is a symmetric operator space on $\mathcal{N}$. Then,
$$
E(0,\infty):=\{f\in S_0((0,\infty)):\ f^*=\mu(x)\ \text{for some}\ x\in \mathcal{E}\},\ \|f\|_E:=\|x\|_{\mathcal{E}}
$$
is a symmetric function space on $(0,\infty)$. It is obvious that $\mathcal{E}=E(\mathcal{N},\tau)$.

We are now fully equipped to provide a full analogue of Corollaries \ref{c1},\ref{c11}.


\begin{corollary}
\label{c3} Let $\mathcal{M}$ be a semi-finite $W^*$ -factor and
let $\mathcal{E}$ be a symmetric operator space. Fix $a=a^*\in
S(\mathcal{M})$ and consider inner derivation $\delta=\delta_a$ on
the algebra $S(\mathcal{M})$ given by $\delta(x)=[a,x]$, $x\in
S(\mathcal{M})$. If $\delta(\mathcal{M})\subseteq \mathcal{E}$,
then there exists $d\in \mathcal{E}$ such that $\delta(x)=[d,x]$.
Furthermore, $\|\delta\|_{\mathcal{M}\rightarrow
\mathcal{E}}<\infty$ and the element $d\in \mathcal{E}$ can be
chosen so that
$\|d\|_{\mathcal{E}}\leq\|\delta\|_{\mathcal{M}\rightarrow
\mathcal{E}}$.
\end{corollary}

\begin{proof}

The existence of $d\in \mathcal{E}$ such that $\delta(x)=[d,x]$
follows from Corollary \ref{c2}. Now, if $u\in U(\mathcal{M})$,
then $\|\delta(u)\|_{\mathcal{E}}=\|du-ud\|_{\mathcal{E}}\leq
\|du\|_{\mathcal{E}}+\|ud\|_{\mathcal{E}}=2\|d\|_{\mathcal{E}}$.
Hence, if $x\in \mathcal{M}_1=\{x\in\mathcal{M}:\ \|x\|\leq 1\}$,
then $x=\sum_{i=1}^4 \alpha_i u_i$, where $u_i\in U(\mathcal{M})$
and $|\alpha_i|\leq 1$ for $i=1,2,3,4$, and so
$\|\delta(x)\|_{\mathcal{E}}\leq \sum_{i=1}^4\|\delta(\alpha_i
u_i)\|_{\mathcal{E}}\leq 8\|d\|_{\mathcal{E}}$, that is
$\|\delta\|_{\mathcal{M}\rightarrow \mathcal{E}}\leq
8\|d\|_{\mathcal{E}}<\infty$.

The final assertion is established exactly as in the proof of
Corollary \ref{c11}.
\end{proof}

An interesting illustration of the result above can be obtained
already for the situation when the space $E$ is given by the norm
closure of the subspace $L_1\cap L_{\infty}$ in the space
$L_1+L_{\infty}$. In this case, the space $\mathcal{E}=E(\mathcal{M},\tau)$
can be equivalently described as the set of all $x\in
L_1+L_{\infty}(\mathcal{M},\tau)$ such that $\lim _{t\to \infty} \mu_t(x)=0$. This space is a natural counterpart of the ideal $\mathcal{K}_{\mathcal {M}}$ of $\tau$-compact operators in $\mathcal {M}$.

%
%
%
%
%
%
%
%
%
%
%

\end{document}